\documentclass[11pt,a4paper]{amsart}
\usepackage[utf8]{inputenc}
\usepackage[english]{babel}
\usepackage[T1]{fontenc}
\usepackage{url}
\usepackage{mathrsfs}
\usepackage{tikz}

\title[Hyperbolic links are not generic]{Hyperbolic links are not generic}
\author{Andrei V. Malyutin}
\thanks{The research is supported by the Foundation for the Advancement of Theoretical Physics and Mathematics ``BASIS''}

\address{St.\,Petersburg Department of 
Steklov Institute of Mathematics\\
St.~Petersburg State University}
\email{malyutin@pdmi.ras.ru}

\def \crn {\operatorname{cr}}

\def \clos {\operatorname{clos}}
\def \card {\operatorname{card}}

\newcommand \sm {\setminus}
\newcommand \R  {\mathbb R}
\newcommand \dd {\partial}
\newcommand \inr{\operatorname{int}}

\newcommand \be     {\begin{equation}}
\newcommand \ee     {\end{equation}}

\newtheorem{thm}{Theorem}
\newtheorem{prop}{Proposition}
\newtheorem{lem}{Lemma}

\newtheorem{claim}{Claim}

\newtheorem{cor}{Corollary}

\theoremstyle{definition}
\newtheorem{definition}{Definition}

\theoremstyle{remark}
\newtheorem{remark}{Remark}

\begin{document}
\maketitle

\begin{abstract}
We show that if $K$ is a nontrivial knot then the proportion of satellites of~$K$ among all of the prime non-split links of $n$ or fewer crossings does not converge to $0$ as $n$ approaches infinity.
This implies in particular that the proportion of hyperbolic links among all of the prime non-split links of $n$ or fewer crossings does not converge to 1 as $n$ approaches infinity. We consider unoriented link types.
\end{abstract} 

\maketitle

\section{Introduction}

The concept of hyperbolicity has many incarnations and relates 
to a series of interesting phenomena in various areas of mathematics.
In geometric topology, hyperbolicity appears in particular in Thurston's classifications.
For instance, Thurston's trichotomy on knot complements says that 
the set of knots splits into  
the classes of satellite, torus, and hyperbolic ones.
Similar classifications are known for links, $3$-manifolds, surfaces and surface automorphisms, groups acting on $1$-dimensional manifolds, etc. 
These classifications sound counter-intuitive because each of the classes involved looks rather special and scarce at first glance, while the union of these seemingly insufficient classes forms the whole thing. 
In~some cases, the answer to this paradox is provided by hyperbolicity 
in the sense that hyperbolic (or hyperbolic-like) objects turn out to be generic. Well-known examples here are hyperbolic surfaces and pseudo-Anosov automorphisms in mapping class groups of surfaces. See~\cite{Mal18} for more details and references.

The present paper deals with the case of links.
Statistics show that the overwhelming majority of prime knots and links of small complexity is of hyperbolic type (see \cite{HTW98, Bur18, Mal18}), 
and it was widely believed for some time that generic prime knots and links are hyperbolic (see, e.\,g., \cite{Ad94, Rat06}).
However, in the cases of knots, links, and $3$-manifolds 
there is a counter-argument to the conjecture that hyperbolic elements are generic. 
The key point of this counter-argument is that the satellite structures in the mentioned cases can be local in a certain sense. 
Presumably, this locality property may or may not appear
in different representations of knots and links, so that distinct types of knots/links can be generic with respect to different representations and complexity measures. 
In~this paper, we consider the crossing number as a complexity measure for links and prove the following theorem.

\begin{thm}\label{TSat}
If $K$ is a nontrivial knot, then the proportion of satellites of~$K$ among all of the prime non-split links of $n$ or fewer crossings does not converge to~$0$ as $n$ approaches infinity. More precisely, if $\crn(K)$ is the crossing number of~$K$, $P_n$ is the number of prime non-split links of $n$ or fewer crossings, and $S_n(K)$ is the number of satellites of~$K$ among prime non-split links of $n$ or fewer crossings, then we have
\begin{equation*}
\limsup_{n\to\infty}\frac{S_n(K)}{P_n}~>~\frac{1}{1+(10.4)^{6(4\crn(K)+1)}} ~>~ 10^{-26\crn(K)}.
\end{equation*}
Furthermore, if $K$ is prime, then we have
\begin{equation*}
\limsup_{n\to\infty}\frac{S_n(K)}{P_n}~>~\frac{1}{1+(10.4)^{6\crn(K)}} ~>~ 10^{-7\crn(K)}.
\end{equation*}
\end{thm}

In the case of the trefoil knot~$3_1$ ($\crn(K)=3$), Theorem~\ref{TSat} yields the following.

\begin{cor}\label{CorHyp}
The proportion of hyperbolic links among all of the prime non-split links of $n$ or fewer crossings does not converge to 1 as $n$ approaches infinity.
More precisely, let $P_n$ {\rm(}\/resp., $H_n$, $S_n$\/{\rm)} denote the number of prime non-split {\rm(}\/resp., hyperbolic, prime non-split satellite\/{\rm)} links of~$n$ or fewer crossings.
Then
\begin{equation*}
\limsup_{n\to\infty}\frac{S_n}{P_n}~>~\frac{1}{1+(10.4)^{18}}~>~10^{-19},
\end{equation*}
and therefore
\begin{equation*}
\liminf_{n\to\infty}\frac{H_n}{P_n}~<~1-10^{-19}.
\end{equation*}
\end{cor}

\begin{remark}
Theorem~\ref{TSat} does not imply that the proportion of hyperbolic knots among all of the prime non-split \underline{knots} of $n$ or fewer crossings does not converge to 1 as $n$ approaches infinity.
\end{remark}

\begin{remark}
A certain modification of our method strengthens the estimates of Theorem~\ref{TSat} and Corollary~\ref{CorHyp}. In particular, the constant $10^{-19}$ of Corollary~\ref{CorHyp} can be replaced with $10^{-13}$, and the constants $10^{-26}$ and $10^{-7}$ of Theorem~\ref{TSat} can be replaced with $10^{-17}$ and $10^{-5}$, respectively.
\end{remark}

The paper is organized as follows. 
Section~\ref{sec:main} describes the idea of the proof of Theorem~\ref{TSat} and 
presents a series of propositions involved in this proof. 
The subsequent sections contain proofs of these propositions.
Our constructions should be interpreted as being in either the PL or smooth category. For standard definitions, we mostly use the conventions of \cite{BZ06} and~\cite{BZH13}. 
There will be a certain abuse of language in order to avoid complicating the notation. 
In particular, a \emph{link} will be a smooth compact one-dimensional submanifold in the 3-sphere~$S^3$ or in~$\R^3$, a~pair $(S^3, L)$, or a class of homeomorphic pairs (cf.~\cite[p.~1]{BZ06}).
No orientations on links and spaces are placed if not otherwise stated.
In particular, when counting links, we consider a link as a class of homeomorphic non-oriented pairs~$(S^3,L)$.

\section{The idea of the proof of Theorem~\ref{TSat}}
\label{sec:main}

A central concept of our proof of Theorem~\ref{TSat} is that of~$K$-en\-tangle\-ments.

\begin{definition}[Entanglements and disentanglements of links]
\label{def:entang}
We recall that the \emph{wrapping number} of a link~$X$ in a solid torus~$U$ is the minimum number of intersections of~$X$ with any meridional disk of~$U$.
If~$U$ is embedded in the $3$-sphere~$S^3$, then there exists an essential curve in~$\partial U$ that bounds a $2$-sided surface in $S^3 \setminus \inr(U)$ (a~\emph{Seifert surface});
this curve is unique up to isotopy on~$\partial U$ and is called a \emph{longitude} of $U$ in~$S^3$ (see, e.\,g., \cite[Theorem~3.1]{BZ06}).
Now, let $L$ be a link in the $3$-sphere $S^3$, and let $V$ be an unknotted solid torus in~$S^3$ such that $L$ is contained in the interior of~$V$. 
Let $W$ be a tubular neighborhood of a knot~$K$ in~$S^3$, and let $\psi\colon V\to W$ be a homeomorphism.
If the wrapping number of~$L$ in~$V$ is nonzero 
then the link $\psi(L)\subset S^3$ is a \emph{satellite} of~$K$.
If the wrapping number of~$L$ in~$V$ is at least~$2$
and $\psi$ is \emph{untwisted} in the sense that it maps a longitude of~$V$ to a longitude of~$W$,
we say that $\psi(L)$ is a \emph{$K$-en\-tangle\-ment} for~$L$ and $L$ is a \emph{$K$-disentanglement} for~$\psi(L)$.
If, in addition, a pair of distinct components of~$L$ have nonzero wrapping number in~$V$ each, then we say that $\psi(L)$ is a \emph{reliable $K$-en\-tangle\-ment} for~$L$.
\end{definition}

\begin{remark}
Definitions obviously imply that a prime non-split link~$L$ is a $K$-en\-tangle\-ment if and only if $L$ is a proper satellite of~$K$.
\end{remark}

Another concept playing a key role in our proof is as follows.

\begin{definition}[Regular knots; see \cite{Mal18}]
If $P$ is a knot and $x$ is a real number,
we say that $P$ is \emph{$x$-regular} if we have $x\cdot\crn(P)\le\crn(K)$ 
whenever $P$ is a factor of a knot~$K$.
If there exists a knot $K$ such that $P$ is a factor of $K$ and $\crn(K)< x\cdot\crn(P)$, we say that $P$ is \emph{non-$x$-regular}.
\end{definition}

\begin{remark}
If the conjecture that the crossing number of knots is additive with respect to connected sum is true, then each knot is $1$-regular.
Results of \cite{La09} imply that each knot is $\frac1{152}$-regular.
If $x<y$ then each $y$-regular knot is $x$-regular.
See \cite{Mal18} for more details.
\end{remark}

\begin{prop}\label{lem:main}
Let $K$ be a nontrivial knot. 
\begin{itemize}
\item[\rm(i)] 
All $K$-en\-tangle\-ments of non-split links are non-split.
\item[\rm(ii)] 
All reliable $K$-en\-tangle\-ments of prime non-split links are prime.
\item[\rm(iii)] 
If $K$ is prime, then the $K$-en\-tangle\-ments over distinct links $L_1$ and~$L_2$ are distinct whenever $L_1$ and~$L_2$ are not $K$-en\-tangle\-ments.
\item[\rm(iv)] If a prime non-split link~$L$ is not a non-$\frac23$-regular knot, then there exists a \underline{prime} $K$-en\-tangle\-ment $L'$ for~$L$ with $\crn(L')\le \crn(L)+6\crn(K)$. 
\end{itemize}
\end{prop}

\begin{prop}
\label{pr:NxRegular}
The proportion of non-$\frac34$-regular prime knots among all of the prime non-split links of $n$ or fewer crossings converges to $0$ as $n$ approaches infinity.
\end{prop}

Propositions~\ref{lem:main} and~\ref{pr:NxRegular} will be proved in subsequent sections. 
Now, we deduce Theorem~{\rm\ref{TSat}} from these propositions. 
The argument is similar to that of Proposition~3.6 in~\cite{Mal18}. 

\begin{proof}[Proof of Theorem~{\rm\ref{TSat}}]
We use the following notation:
\begin{itemize}
\item[$P_n$] is the number of prime non-split links of $n$ or fewer crossings,
\item[$N_n$] is the number of non-$\frac23$-regular knots of $n$ or fewer crossings, and
\item[$S_n(K)$] is the number of satellites of~$K$ among prime non-split links of $n$ or fewer crossings.
\end{itemize}
  
First, we consider the case where $K$ is prime.
By definition, each $K$-en\-tangle\-ment is a satellite of~$K$.
Then assertions~(iii) and~(iv) of Proposition~\ref{lem:main} imply that for all~$n$ we have 
$$
S_{n+6\crn(K)}(K)~\ge~P_n-S_n(K)-N_n.
$$
This is equivalent to the following inequality
\begin{equation*}
\frac{S_{n+6\crn(K)}(K)}{P_{n+6\crn(K)}}~\ge~ \left(1-\frac{S_n(K)+N_n}{P_n}\right)\cdot\frac{P_n}{P_{n+6\crn(K)}},
\end{equation*}
which implies that
\begin{multline}
\label{eq:PHPHlsli}
\limsup_{n\to\infty}\frac{S_{n+6\crn(K)}(K)}{P_{n+6\crn(K)}}~\ge~ 
\liminf_{n\to\infty}\left(1-\frac{S_n(K)+N_n}{P_n}\right)\cdot\limsup_{n\to\infty}\frac{P_n}{P_{n+6\crn(K)}}\\
~=~ \left(1-\limsup_{n\to\infty}\frac{S_n(K)+N_n}{P_n}\right)\cdot\limsup_{n\to\infty}\frac{P_n}{P_{n+6\crn(K)}}.
\end{multline}
Since we have 
$$\limsup\limits_{n\to\infty}\frac{S_{n+6\crn(K)}(K)}{P_{n+6\crn(K)}}~=~\limsup\limits_{n\to\infty}\frac{S_n(K)}{P_n}$$
while by Proposition~\ref{pr:NxRegular} we have
$$
\limsup_{n\to\infty}\frac{N_n}{P_n}~=~\lim_{n\to\infty}\frac{N_n}{P_n}~=~0,
$$
it follows that \eqref{eq:PHPHlsli} is equivalent to the inequality
\begin{equation}
\label{eq:PHPHxls}
\limsup_{n\to\infty}\frac{S_n(K)}{P_n}~\ge~\frac{1}{1+\left({\limsup\limits_{n\to\infty}\frac{P_n}{P_{n+6\crn(K)}}}\right)^{-1}}~=~\frac{1}{1+{\liminf\limits_{n\to\infty}\frac{P_{n+6\crn(K)}}{P_n}}}.
\end{equation}
We observe that
\begin{multline*}
\liminf\limits_{n\to\infty}\frac{P_{n+6\crn(K)}}{P_n}
~\le~
\limsup_{n\to\infty}\left(\sqrt[n]{P_n}\right)^{6\crn(K)}\\
~\le~
\limsup_{n\to\infty}\left(\sqrt[n]{n\cdot p_n}\right)^{6\crn(K)}
~=~
\limsup_{n\to\infty}\left(\sqrt[n]{p_n}\right)^{6\crn(K)},
\end{multline*}
where $p_n$ is the number of prime non-split links of precisely~$n$ crossings (so that $P_n=p_1+\cdots+p_n$).
An estimate for $\limsup\limits_{n\to\infty}\sqrt[n]{p_n}$ comes from a paper of Alexander Stoimenow~\cite{St04}, whose results imply that
$$
\limsup_{n\to\infty}\sqrt[n]{p_n}~\le~\frac{\sqrt{13681}+91}{20}~<~10.4.
$$
Therefore, we have
\begin{equation}
\label{eq:17-1}
\liminf_{n\to\infty}\frac{P_{n+6\crn(K)}}{P_n}~<~(10.4)^{6\crn(K)}.
\end{equation}
Plugging~\eqref{eq:17-1} into~\eqref{eq:PHPHxls} produces the required
$$
\limsup_{n\to\infty}\frac{S_n(K)}{P_n} ~>~
\frac{1}{1+(10.4)^{6\crn(K)}} ~>~ 10^{-7\crn(K)}.
$$

In the case where $K$ is composite, we apply \emph{cable knots}.
A~knot~$Y$ is said to be a \emph{cable knot} over a knot~$X$ if $Y$ is a proper satellite of $X$ contained in the boundary torus~$\dd V$ of a tubular neighborhood~$V$ of a representative of~$X$.
We recall that 
\begin{itemize}
\item each nontrivial knot~$X$ has a cable knot~$Y$ over~$X$ with $\crn(Y)\le 4\crn(X)+1$ ($Y$~can be obtained by ``doubling'' a minimal diagram of~$X$ and ``adding'' a crossing);
\item each cable knot over~$X$ is a prime satellite of~$X$ (see \cite[p.~250, Satz~4]{Schu53}, \cite[Cor.~2]{Gra91}).
\end{itemize}
Let $K_2$ be a cable knot over~$K$ with $\crn(K_2)\le 4\crn(K)+1$.
Since $K_2$ is prime, by the first part of the present proof we have 
$$
\limsup_{n\to\infty}\frac{S_n(K_2)}{P_n} ~>~\frac{1}{1+(10.4)^{6\crn(K_2)}}.
$$
Since $K_2$ is a satellite of~$K$, it follows that each satellite of~$K_2$ is a satellite of~$K$, so that we have $S_n(K)\ge S_n(K_2)$.
Thus we have (given that $\crn(K)\ge 6$ because~$K$ is composite)
\begin{multline*}
\limsup_{n\to\infty}\frac{S_n(K)}{P_n} ~\ge~
\limsup_{n\to\infty}\frac{S_n(K_2)}{P_n} ~>~
\frac{1}{1+(10.4)^{6\crn(K_2)}}\\
~\ge~
\frac{1}{1+(10.4)^{6(4\crn(K)+1)}}~\ge~
\frac{1}{1+(10.4)^{24\crn(K)+6}}
\\
~\ge~
\frac{1}{1+(10.4)^{25\crn(K)}}~\ge~
10^{-26\crn(K)}.\qedhere
\end{multline*}
\end{proof}

\section{Proof of Assertion (i) of Proposition~\ref{lem:main}}
\label{sec:proof-i}

We show that no en\-tangle\-ment of a non-split link is split.
Suppose to the contrary that a split link~$L$ is a $K$-en\-tangle\-ment for a non-split link~$L'$.
By the definition of entanglements, this implies that there exist an embedded solid torus~$V$ in $S^3$ and a re-embedding $\phi\colon V\to S^3$ such that the following conditions hold:
\begin{itemize}
\item $V$ is a tubular neighborhood of a closed curve representing~$K$;
\item $L$ lies in the interior of~$V$ and the wrapping number of~$L$ in~$V$ is at least~$2$;
\item the solid torus $\phi(V)$ is unknotted;
\item $\phi$ maps a longitude of $V$ to a longitude of $\phi(V)$; 
\item we have $\phi(L)=L'$.
\end{itemize}

We observe that $V$ is nontrivially knotted because otherwise we have $L=L'$.

\begin{claim}
\label{cl:V-incompressible}
The torus $\dd V$ is incompressible in~$S^3\setminus L$.
\end{claim}

It is well known that the complement of any nontrivial knot is boundary incompressible while the only compressing disks in a solid torus are meridional ones (see, e.\,g., \cite[Propositions~3.10 and~3.12 and Exercise E~2.9]{BZH13} and \cite[p.~15, Example before Lemma 1.10]{Hat00}). Since the wrapping number of $L$ in $V$ is nonzero, it follows that no compressing disk for~$\dd V$ is contained in~$V\sm L$. 
This proves the claim.

Now, let $S^2$ be a splitting sphere for~$L$ in~$S^3\sm L$.
It may be assumed that $S^2$ intersects $\dd V$ transversely in simple closed curves and that the intersection $S^2\cap \dd V$ has the minimal number of components among all splitting spheres for~$L$. 

If $S^2\cap \dd V=\emptyset$, then $S^2$ lies in~$V$ because $V$ contains~$L$.
This readily implies that $\phi(S^2)$ splits~$\phi(L)=L'$ because~$S^2$ bounds a ball in~$V$. A contradiction.

If $S^2\cap \dd V\neq\emptyset$, then a component of~$S^2\sm \dd V$ is a disk.
We denote the closure of this disk by~$D$. 
Since $\dd V$ is incompressible in~$S^3\setminus L$ (see Claim~\ref{cl:V-incompressible}), it follows that $\dd D$ bounds a disk, say~$d$, in~$\dd V$.

If the sphere $D\cup d$ splits~$L$, then $D\cup d$ is isotopic to a splitting sphere for~$L$ that does not intersect~$\dd V$, which contradicts the assumed minimality of the number of components in $S^2\cap \dd V$.

If $D\cup d$ does not split~$L$, then $D\cup d$ bounds a ball in $S^3\sm L$.
Therefore, by an isotopy of $S^2$ in a neighborhood of this ball, we can eliminate the component $\dd D$ of $S^2\cap \dd V$. This contradicts the same assumption on the minimality of the number of components in $S^2\cap \dd V$.
This contradiction completes the proof of assertion~(i) of Proposition~\ref{lem:main}.


\section{Proof of Assertion (ii) of Proposition~\ref{lem:main}}
\label{sec:proof-ii}

We show that the reliable $K$-en\-tangle\-ments of prime non-split links are prime.

Suppose to the contrary that a composite link~$L$ is a reliable $K$-en\-tangle\-ment for a prime non-split link~$L'$.
By the definition of reliable entanglements, this implies that there exist an embedded solid torus~$V$ in $S^3$ and a re-embedding $\phi\colon V\to S^3$ such that the following conditions hold:
\begin{itemize}
\item $V$ is a tubular neighborhood of a closed curve representing~$K$;
\item $L$ lies in the interior of~$V$ and a pair of distinct components of~$L$ have nonzero wrapping number in~$V$ each;
\item the solid torus $\phi(V)$ is unknotted;
\item $\phi$ maps a longitude of $V$ to a longitude of $\phi(V)$; 
\item we have $\phi(L)=L'$.
\end{itemize}

We observe that $V$ is nontrivially knotted because otherwise we have ${L=L'}$.
Since~$L$ is composite, it follows that $S^3$ contains a sphere~$S^2$ intersecting~$L$  transversely in two points such that neither of~$(B_1,B_1\cap L)$ and~$(B_2,B_2\cap L)$, where $B_1$ and $B_2$ are the closures of the components of~$S^3\sm S^2$, is a trivial one-string tangle.
It may be assumed that $S^2$ intersects $\dd V$ transversely in simple closed curves and that the intersection $S^2\cap \dd V$ has the minimal number of components among all decomposition spheres for~$L$. 

If $S^2\cap \dd V=\emptyset$, then $S^2$ lies in~$V$ because $V$ contains~$L$.
Then~$S^2$ bounds a ball, say~$B_1$, in~$V$.
We observe that at least two components of~$L$ intersect $V\sm B_1$ because  
a pair of distinct components of~$L$ have nonzero wrapping number in~$V$ while if a component is contained in~$B_1$ then it has zero wrapping number in~$V$.
Therefore, the tangle $(\phi(B_1),\phi(B_1)\cap L')$ is not a trivial one-string tangle because it is equivalent to~$(B_1,B_1\cap L)$ and the tangle 
$$(\clos(S^3\sm\phi(B_1)),\clos(S^3\sm\phi(B_1))\cap L')$$ is not a trivial one-string tangle because $\clos(S^3\sm\phi(B_1)$ intersects at least two components of~$L'$. This means that $L'$ is composite in contradiction to our assumptions.


If $S^2\cap \dd V\neq\emptyset$, then a pair of distinct component of~$S^2\sm \dd V$ are disks. We denote the closures of these disks by~$D_1$ and~$D_2$. 
Since $S^2$ intersects $L$ in two points, it follows that at least one of these disks intersects $L$ in at most one point.
Without loss of generality we assume that $D_1$ intersects $L$ in at most one point. Since $\dd V$ is incompressible in~$S^3\setminus L$ (see Claim~\ref{cl:V-incompressible}), it follows that $\dd D_1$ bounds a disk, say~$d$, in~$\dd V$. Since~$D_1\cup d$ is a sphere, $D_1$ intersects $L$ in at most one point, and $d\subset \dd V$ does not intersect~$L$, it follows that~$D_1\cup d$
does not intersect~$L$. 
By assertion~(i) of the proposition, $L$ is non-split.
Therefore, $D_1\cup d$ bounds a ball in $S^3\sm L$.
Consequently, by an isotopy of $S^2$ in a neighborhood of this ball, we can eliminate the component $\dd D_1$ of $S^2\cap \dd V$, which contradicts the assumption on the minimality of the number of components in $S^2\cap \dd V$.
This contradiction completes the proof of assertion~(ii) of Proposition~\ref{lem:main}.

\section{Proof of Assertion (iii) of Proposition~\ref{lem:main}}
\label{sec:proof-iii}

We show that if $K$ is a prime knot then the $K$-en\-tangle\-ments over distinct links that are not $K$-en\-tangle\-ments are distinct.

Suppose to the contrary that there exist links $L_1$, $L_2$, and~$L$ such that
$L_1\neq L_2$, 
$L$ is a $K$-en\-tangle\-ment for each of~$L_1$ and~$L_2$, and 
neither of~$L_1$ and~$L_2$ is a $K$-en\-tangle\-ment.
By the definition of entanglements, this means that there exist embedded solid tori $V_1$ and $V_2$ in $S^3$ and re-embeddings $\phi_1\colon V_1\to S^3$ and $\phi_2\colon V_2\to S^3$ such that, for each $i\in\{1,2\}$, the following conditions hold:
\begin{itemize}
\item $V_i$ is a tubular neighborhood of a closed curve representing~$K$;
\item $L$ lies in the interior of $V_i$ and the wrapping number of $L$ in $V_i$ is at least~$2$;
\item the solid torus $\phi_i(V_i)$ is unknotted;
\item $\phi_i$ maps a longitude of $V_i$ to a longitude of $\phi_i(V_i)$;
\item we have $\phi_i(L)=L_i$.
\end{itemize}

\begin{claim}
\label{cl:non-split}
If $L$ is split then it has a (unique) non-split component that is a $K$-en\-tangle\-ment for two distinct links that are not $K$-en\-tangle\-ments.
\end{claim}

Indeed, the proof of assertion (i) of Proposition~\ref{lem:main} shows that if $L$ is split then $V_1$ contains a ball~$B$ whose boundary splits~$L$.
Since $B\cap L$ has zero wrapping number in~$V_1$, it follows that the wrapping number of $L\sm B$ in $V_1$ equals the wrapping number of $L$ in~$V_1$. Therefore, $L\sm B$ is a $K$-en\-tangle\-ment for $\phi_1(L\sm B)$ and $L_1$ is the union of $\phi_1(L\sm B)$ and $\phi_1(L\cap B)=L\cap B$.
If $L\sm B$ is split, we repeat the argument. 
Induction shows that $L$ has a (unique) non-split component $L'_1$ that is a $K$-en\-tangle\-ment for~$\phi_1(L'_1)$ and $L_1$ is the union of $\phi_1(L'_1)$ and $\phi_1(L\sm L'_1)=L\sm L'_1$.
Applying this argument for~$V_2$ yields that $L$ has a (unique) non-split component $L'_2$ that is a $K$-en\-tangle\-ment for~$\phi_2(L'_2)$ and $L_2$ is the union of $\phi_2(L'_2)$ and $\phi_2(L\sm L'_2)=L\sm L'_2$.
It remains to notice that $L'_1=L'_2$ because otherwise $L_1$ is a $K$-en\-tangle\-ment of a union of $\phi_1(L'_1)$, $\phi_2(L'_2)$, and $L\sm (L'_1\cup L'_2)$. The claim is proved.

Thus, it is enough to consider the case where $L$ is nonsplit, which we assume in the rest of the proof.

\begin{claim}
\label{cl:Vi-incompressible}
The tori $\partial V_1$ and $\partial V_2$ are both incompressible in~$S^3\setminus L$.
\end{claim}

See Claim~\ref{cl:V-incompressible}.

\begin{claim}
\label{cl:isotV1V2j}
There is no isotopy between $\partial V_1$ and $\partial V_2$ in $S^3\setminus L$.
\end{claim}

Suppose to the contrary that such an isotopy exists.
The classical Isotopy Extension Theorem (for smooth manifolds) says that  
if $A$ is a compact submanifold of a manifold~$M$ and $F\colon A\times I\to M$ is an isotopy of~$A$ with $F(A\times I)\subset \inr(M)$,
then $F$ extends to an ambient isotopy (i.\,e., a diffeotopy of $M$) having compact support (see, e.\,g., \cite[p.\,179]{Hir76}).
Applying this theorem to the isotopy between $\partial V_1$ and $\partial V_2$ in $S^3\setminus L$ implies that
there exists an ambient isotopy of $S^3$, fixing $L$ pointwise, that moves $\partial V_1$ to $\partial V_2$.
This yields an isotopy between $V_1$ and $V_2$ that fixes $L$ pointwise.
Then the triples $(V_1,L,\ell_1)$ and $(V_2,L,\ell_2)$, where $\ell_i$ is a longitude of~$V_i$, $i=1,2$, are homeomorphic, i.\,e.,
there exists a homeomorphism $\tau\colon V_1\to V_2$ such that $\tau(L)=L$ and $\tau(\ell_1)=\ell_2$.
This implies that the pairs $(S^3,\phi_1(L))$ and $(S^3,\phi_2(L))$ are homeomorphic. 
Indeed, we observe that $(S^3,\phi_i(L))$ is obtained from $(V_i, L)$ by a \emph{Dehn filling} along $\ell_i$, that is, 
$(S^3,\phi_i(L))$ is obtained by attaching a solid torus $V$ to $V_i$ by a gluing homeomorphism $\sigma_i\colon \partial V\to \partial V_i$ such that $\sigma_i^{-1}(\ell_i)$ bounds a meridional disk of~$V$.  
Thus, the homeomorphism $\tau\colon V_1\to V_2$ extends to a homeomorphism $S^3\to S^3$ that maps $\phi_1(L)$ to $\phi_2(L)$.
This contradicts the assumption that the links $L_1$ and $L_2$ are distinct because we have $\phi_i(L)=L_i$ by construction. 
Claim~\ref{cl:isotV1V2j} is proved.

\begin{claim}
\label{cl:noV1subpsetV}
There is no isotopy of $S^3$ that fixes $L$ pointwise and moves $V_1$ to a position where either $V_1\subset V_2$ or $V_2\subset V_1$.
\end{claim}

Suppose to the contrary that such an isotopy exists.
We consider the case with $V_1\subset V_2$. 
(The case with $V_2\subset V_1$ is equivalent.)
Since there is no isotopy between $\partial V_1$ and $\partial V_2$ in $\clos(S^3\sm V_1)\subset S^3\setminus L$ (Claim~\ref{cl:isotV1V2j}), it then follows that $\dd V_2$ is an essential (incompressible, non-boundary parallel) torus in~$\clos(S^3\sm V_1)$.
This means that $K$ is a nontrivial satellite (and companion) of itself.
However, this is impossible.
Indeed, let $f\colon \clos(S^3\sm V_1)\to \clos(S^3\sm V_2)$ be a homeomorphism. Then the infinite sequence 
$$
f(\dd V_2), \quad f(f(\dd V_2)), \quad f(f(f(\dd V_2))), \quad \dots
$$ 
yields an infinite collection of disjoint incompressible connected surfaces which are pairwise nonparallel. 
This contradicts the well-known Finiteness Principles (Haken's Parallelism Principle) saying that any Haken manifold contains only a bounded number of disjoint incompressible boundary incompressible connected surfaces which are pairwise nonparallel and that the number of equivalence classes of incompressible tori is finite for any Haken manifold (see, e.\,g., \cite[Theorems 6.3.10 and 6.4.44]{Mat07}; \cite[p.~140]{Hem76} and \cite[p.~28]{BS16} for the case of link complements).

We split the further proof in the following two cases:

\begin{itemize}
\item[\bf Case A.] There exists an isotopy of $\partial V_1$ in $S^3\setminus L$ that moves $\partial V_1$ to a position where $\partial V_1 \cap \partial V_2=\emptyset$.
\item[\bf Case B.] No isotopy of $\partial V_1$ in $S^3\setminus L$ moves $\partial V_1$ to a position where $\partial V_1 \cap \partial V_2=\emptyset$.
\end{itemize}

\subsection{Proof in Case A}

In Case A, the Isotopy Extension Theorem (see above) implies that
there exists an ambient isotopy of $S^3$, fixing $L$ pointwise, that moves $V_1$ to a position where $\partial V_1\cap \partial V_2=\emptyset$.
Thus, we can assume without loss of generality that $\partial V_1 \cap \partial V_2 = \emptyset$ (while $V_1$ and $V_2$ satisfy all of the properties listed at the beginning of the proof). Now, let $M_1$ and $M_2$ denote the closures of the complements $S^3\setminus V_1$ and $S^3\setminus V_2$, respectively.
Then Claim~\ref{cl:noV1subpsetV} implies the following.

\begin{claim}
\label{cl:disj}
In Case A, $M_1$ and $M_2$ are disjoint.
\end{claim}

Another fact that we need is implied by the following proposition.

\begin{prop}
\label{prop:Budney}
Let $C_1$, $C_2$, \dots, $C_n$ be $n$ disjoint submanifolds of $S^3$ 
such that $K_i =\clos(S^3 \setminus C_i)$ is a nontrivially embedded solid torus in~$S^3$ for all $i \in \{1, 2, \dots , n\}$. 
Then there exist $n$ disjointly embedded $3$-balls $B_1$, $B_2$, \dots, $B_n \subset S^3$ such that $C_i \subset B_i$ for all $i \in \{1, 2, \dots , n\}$. 
Moreover, each $B_i$ can be chosen to be $C_i$ union a $2$-handle which is a tubular neighborhood of a meridional disk for $K_i$.
\end{prop}

\begin{proof}
See {\cite[Proposition\,2.1, p.~324]{Bud06}} and references therein for earlier proofs.
\end{proof}

Applying Proposition~\ref{prop:Budney} to $M_1$ and $M_2$, we obtain the following claim.

\begin{claim}
\label{cl:2121}
In Case A, there exists a meridional disk $D_2$ for $V_2$ such that $D_2 \subset V_1$.
\end{claim}

Now, since $M_1$ and $M_2$ are disjoint, we have $M_2\subset V_1$.
Therefore, the image $\phi_1(M_2)$ is well defined.
We consider the complement $W:=S^3\setminus \phi_1 (\inr(M_2))$.
By Alexander's theorem on embedded tori in~$S^3$, we observe that $W$ is a knotted solid torus because we know that the boundary $\partial W= \partial \phi_1(M_2) = \phi_1(\partial M_2)=\phi_1(\partial V_2)$ is a torus while the complement $S^3\setminus W=\phi_1 (\inr(M_2))$ is homeomorphic to $\inr(M_2)$, which is the complement of the knotted solid torus~$V_2$. 
By the Gordon--Luecke theorem (see~\cite{GL89,GL89'}), the knot complement determines the knot. 
Therefore, $W$ is a tubular neighborhood of a closed curve representing~$K$.
We see that $W$ contains $\phi_1(L)$ by construction.
Finally, we see that the wrapping number of $\phi_1(L)$ in~$W$ is equal to the wrapping number of $L$ in~$V_2$ because there exists a meridional disk $D_2$ for $V_2$ such that $D_2 \subset V_1$ (see~Claim~\ref{cl:2121}), so that $\phi_1$ maps $D_2$ to a meridional disk of~$W$. 
Therefore, $\phi_1(L)$ is contained in a $K$-knotted solid torus~$W$ and the wrapping number of $\phi_1(L)$ in~$W$ is at least~$2$.
This means that $\phi_1(L)=L_1$ is a $K$-en\-tangle\-ment, which contradicts the assumption.
This contradiction completes the proof of assertion~(iii) of Proposition~\ref{lem:main} in Case~A.

\subsection{Proof in Case B}

Our proof for Case~B is in the standard framework of JSJ decomposition theory (see, e.\,g., \cite{Hat00, Bud06, Mat07} and references therein). 
Let $U_L$ be a closed tubular neighborhood of~$L$ in~$S^3$ such that $\dd V_1$ and $\dd V_2$ are both contained in $S^3\sm U_L$ and let~$C_L=\clos(S^3\sm U_L)$.
Then $C_L$ is prime because we assume that $L$ is non-split.
The JSJ-decomposition theorem implies that $C_L$ contains a (unique up to isotopy) collection of embedded, incompressible tori $T$ such that 
\begin{itemize}
\item if one removes an open tubular neighbourhood $U_T$ of $T$ from $C_L$ then the resulting manifold $C_L\sm U_T$ is a disjoint union of Seifert-fibered and atoroidal manifolds;
\item no proper subcollection of $T$ has the preceding property.
\end{itemize}
If $F$ is an incompressible surface in a prime $3$-manifold $M$ and $T$ is a collection of tori in~$M$ that gives the JSJ-decomposition, then there is an isotopy of $F$ in $M$ that moves $F$ to a position where $F \cap T=\emptyset$ (see, e.\,g., \cite[Sec.~6.4]{Mat07}). Since the tori $\partial V_1$ and $\partial V_2$ are both incompressible in~$C_L$ (Claim~\ref{cl:Vi-incompressible}), we can apply the Isotopy Extension Theorem (see above) and assume without loss of generality that neither $\partial V_1$ nor $\partial V_2$ intersects~$U_T$.
Then the initial assumption of Case~B implies 

\begin{claim}
\label{cl:Seifert}
In Case B, $\partial V_1$ and $\partial V_2$ are contained in the same connected component, say~$E$, of~$C_L\sm U_T$ and neither $\partial V_1$ nor $\partial V_2$ is boundary parallel in~$E$. Since $\partial V_1$ and $\partial V_2$ are incompressible, this implies that $E$ is a Seifert-fibered manifold.
\end{claim}

The classification of Seifert-fibered manifolds arising as JSJ-chambers for link complements in~$S^3$ is well-known. See~\cite{Bud06}. In particular, the proof of \cite[Proposition~4]{Bud06} implies that one of the following holds:

\begin{itemize}
\item[\bf Case B$_1$.] $E$~fibers, with $2$ singular fibers, over the sphere with $n\ge 1$ holes. 

In this case, $E$ is the complement of a tubular neighborhood of $n$ regular fibers in a Seifert-fibering of~$S^3$. 

\item[\bf Case B$_2$.] $E$~fibers, with~$1$ singular fiber, over the sphere with $n\ge 2$ holes.

In this case, $E$ is the complement of a tubular neighborhood of $n-1$ regular fibers in a 
Seifert-fibering of an embedded solid torus in~$S^3$. 

\item[\bf Case B$_3$.] $E$~fibers, with no singular fibers, over the sphere with $n\ge3$ holes.

In this case, one of the following holds (the subcases are not mutually exclusive):

\begin{itemize}
\item[\bf Case B$_{3a}$.]
$E$ is the complement of a tubular neighborhood of $n-1$ fibers in a trivial 
Seifert-fibering of an embedded 
solid torus in~$S^3$. 

\item[\bf Case B$_{3b}$.] 
$E$ is the complement of a tubular neighborhood of $n-2$ regular fibers and the singular fiber in a Seifert-fibering of an embedded solid torus in~$S^3$. 

\item[\bf Case B$_{3c}$.] 
A component~$Y$ of~$\clos(S^3\sm E)$ is a solid torus such that the meridians of $Y$ are fibers of~$E$. 
\end{itemize}
\end{itemize}

It is well known that any incompressible torus in a connected compact irreducible Seifert-fibered manifold with nonempty boundary is isotopic to a surface that is a union of regular fibers (see, e.\,g., \cite[Proposition~1.11]{Hat00}). Once more, we can apply the Isotopy Extension Theorem (see above) and assume without loss of generality that both $\partial V_1$ and $\partial V_2$ are unions of regular fibers. Then each of $\partial V_1$ and $\partial V_2$ determines a simple closed curve in the base of~$E$.
Let~$\Delta$ denote the base of~$E$ and let $\gamma_1$ and $\gamma_2$ be the curves in~$\Delta$ representing $\partial V_1$ and $\partial V_2$, respectively. 
Before considering Cases B$_1$--B$_{3c}$ separately,
we list several obvious general properties of~$\gamma_1$ and~$\gamma_2$.

\begin{claim}
\label{cl:essential-base}
\begin{itemize}
\item Neither~$\gamma_1$ nor~$\gamma_2$ bounds in~$\Delta$ a disk with $\le 1$ singular point because neither~$\partial V_1$ nor~$\partial V_2$ is compressible.

\item Neither~$\gamma_1$ nor~$\gamma_2$ is isotopic (rel. singular points) to a boundary component of~$\Delta$ because neither~$\partial V_1$ nor~$\partial V_2$ is boundary parallel in~$E$.
\end{itemize}
\end{claim}

We say that a simple closed curve $\gamma$ in~$\Delta$ is $L$-separating if~$L$ intersects both components of~$S^3\sm\tilde\gamma$, where $\tilde\gamma$ is the lift of~$\gamma$ to~$E\subset S^3$.

\begin{claim}
\label{cl:separator}
\begin{itemize}
\item Neither~$\gamma_1$ nor~$\gamma_2$ is $L$-separating.

\item If $\gamma$ is a non-$L$-separating simple closed curve in~$\inr(\Delta)$ and $C_1$, $C_2$ are two $L$-separating curves in~$\Delta\sm\gamma$, then $C_1$ and $C_2$ lie in the same connected component of~$\Delta\sm\gamma$.

\item If two distinct components $X$ and $Y$ of~$\clos(S^3\sm E)$ are solid tori, then the boundary components of~$\Delta$ representing $\dd X$ and $\dd Y$ are both $L$-separating.
\end{itemize}
\end{claim}

Now, we study Cases B$_1$--B$_{3c}$ separately.

\subsubsection{Proof in Case~B$_1$}
In this case, $\Delta$ is a sphere with $n\ge 1$ holes and $2$ singularity points. First, we observe that under our assumptions we have ${n\ge 2}$ because otherwise no simple closed curve in~$\Delta$ satisfies the requirements of Claim~\ref{cl:essential-base}.
Then, since in Case B$_1$ each component of~$\clos(S^3\sm E)$ is a solid tori, it follows by Claim~\ref{cl:separator} that all of the boundary components of~$\dd \Delta$ are $L$-separating and lie in the same component of~$\Delta\sm\gamma_i$. Therefore, another component of~$\Delta\sm\gamma_i$ should contain both of the singularity points. 
It then readily follows that there is an automorphism~$f$ of~$\Delta$ such that 
\begin{itemize}
\item $f$ is an identity on some neighborhoods of the singularity points;
\item $f$ is an identity on $\dd\Delta$;
\item $f(\gamma_1)=\gamma_2$.
\end{itemize}
The lift $\tilde f\colon E\to E$ of~$f$ yields an automorphism~$F$ of~$S^3$ such that 
\begin{itemize}
\item $F$ is an identity on $(S^3\sm E) \supset L$; 
\item $F(\dd V_1)=\dd V_2$.
\end{itemize}

Since $V_1$ and $V_2$ are nontrivially knotted, it follows that $F(V_1)=V_2$.
This means that the triples $(S^3,V_1,L)$ and $(S^3,V_2,L)$ are homeomorphic, which implies that $L_1=L_2$. This contradiction completes the proof in Case~B$_1$.

\subsubsection{Proof in Case~B$_2$}
This case is similar to that of~B$_1$.
In Case~B$_2$, $\Delta$ is a sphere with $n\ge 2$ holes and $1$ singularity point. First, we observe that under our assumptions we have $n\ge 3$ because otherwise no simple closed curve in~$\Delta$ satisfies the requirements of Claim~\ref{cl:essential-base}.
Then, since in Case B$_2$ at least~$n-1$ components of~$\clos(S^3\sm E)$ are solid tori, it follows by Claim~\ref{cl:separator} that the corresponding~$n-1$ components of~$\dd \Delta$ are $L$-separating and lie in the same component of~$\Delta\sm\gamma_i$. Therefore, another component of~$\Delta\sm\gamma_i$ should contain the singularity point and precisely~$1$ component~$C$ of~$\dd\Delta$. 
The remaining argument repeats that of Case~B$_1$ verbatim.

\subsubsection{Proof in Cases~B$_{3a}$ and B$_{3b}$}
In these cases, $\Delta$ is a sphere with $n\ge 3$ holes and no singularity points. We see that, in both these cases, at least~$n-1\ge 2$ components of~$\clos(S^3\sm E)$ are solid tori. Then Claim~\ref{cl:separator} implies that the corresponding~$n-1$ components of~$\dd \Delta$ are $L$-separating and lie in the same component of~$\Delta\sm\gamma_i$.
Therefore, each of~$\gamma_1$ and~$\gamma_2$ is isotopic to the non-$L$-separating component of~$\dd\Delta$.
This contradicts Claim~\ref{cl:essential-base}.

\subsubsection{Proof in Case~B$_{3c}$}
We will prove that in Case~B$_{3c}$, the only possibility for~$V_1$ and~$V_2$ is to represent a composite knot type, in contradiction to the assumption that~$K$ is prime. In fact, this can be easily seen from related geometric constructions presented in~\cite{Bud06}. However, copying and explaining these constructions here would take a lot of place, so we prefer less intuitive but shorter way of proof. 

In Case~B$_{3c}$, a component~$Y$ of~$\clos(S^3\sm E)$ is a solid torus such that the meridians of $Y$ are fibers of~$E$. 
We introduce the following notation. If $\gamma$ is a simple closed curve in~$\Delta$ and $\tilde\gamma$ is the lift of~$\gamma$ to~$E$, we denote by~$Y(\gamma)$ the closure of the connected component of~$S^3\sm \tilde\gamma$ containing $\inr(Y)$. We denote by $\gamma_Y$ be the component of~$\dd\Delta$ representing~$\dd Y$.

\begin{claim}
\label{cl:Ygamma}
If $\gamma$ is a simple closed curve in~$\Delta$, then $Y(\gamma)$ is a solid torus whose meridians are fibers of~$E$. 
\end{claim}

Let $\alpha$ be a simple arc in~$\Delta$ with the endpoints at $\gamma$ and $\gamma_Y$ such that the only intersection points of~$\alpha$ and $\gamma\cup \gamma_Y$ are these endpoints.
Then $\alpha$ lifts to an annulus $\tilde\alpha$, in~$E$, and the endpoint~$\alpha\cap\gamma_Y$ lifts to a meridian of~$Y$.
Attaching an appropriate meridional disk of~$Y$ to~$\tilde\alpha$, 
we obtain a disk, say~$D$.
We observe that $D\cap \dd Y(\gamma)=\dd D$ and that $\dd D$ does not bounds a disc in~$\dd Y(\gamma)$ because if $\dd D$ bounds a disk~$D'$ in~$\dd Y(\gamma)$ then a core curve of~$Y$ would intersect the sphere $D\cup D'$ transversely in a single point, which is impossible.
Therefore, $D$ is a compressing disk for $Y(\gamma)$.
By Alexander's theorem on embedded tori, $Y(\gamma)$ is either a solid torus or a nontrivial knot complement. 
Since nontrivial knot complements are boundary irreducible (see, e.\,g., \cite[Proposition~3.10]{BZH13}), it follows that $Y(\gamma)$ is a solid torus whose meridians are fibers of~$E$. 

\begin{claim}
\label{cl:Ygamma12}
We have $Y(\gamma_1)=V_1$ and $Y(\gamma_2)=V_2$.
\end{claim}

Since $\dd Y$ is incompressible in~$S^3\sm L$, it follows that $Y$ intersects~$L$. Therefore, $V_i$ contains $Y$ because $V_i$ contains~$L$.
Then we have $Y(\gamma_i)=V_i$ by definition of $Y(\gamma_i)$.


\begin{claim}
\label{cl:non-sep-knotted}
If $C$ is a non-$L$-separating component of~$\dd\Delta$, then $Y(C)$ is nontrivially knotted.
\end{claim}

We observe that the lift $\tilde C'$ of a component~$C'$ of~$\dd\Delta$ is incompressible in~$S^3\sm L$. Therefore, if $\tilde C'$ is trivially knotted, both components of~$S^3\sm \tilde C'$ intersect~$L$ so that~$C'$ is $L$-separating.

Without loss of generality we can assume that $\gamma_1$ lies in~$\inr(\Delta)$.

\begin{claim}
\label{cl:non-sep-in-gammai}
If $C$ is a component of~$\dd\Delta$ such that $C$ and $\gamma_Y$ lie in distinct components of~$\Delta\sm\gamma_1$, then $C$ is non-$L$-separating.
\end{claim}

We denote by $\widetilde{C}$ the lift of~$C$ in~$E$ and define a graph~$\Gamma$ as follows. 
\begin{itemize}
\item The vertices of~$\Gamma$ correspond to the connected components of the space~$S^3\sm(\widetilde{C}\cup \dd V_1 \cup \dd Y)$;

\item The vertices of~$\Gamma$ correspond to the elements of the set~$\mathcal{F}:=\{\widetilde{C},\dd V_1,\dd Y\}$; 

\item A~pair of vertices of~$\Gamma$ are connected by an edge if and only if the closures of the components corresponding to these vertices both contain the same torus of~$\mathcal{F}$.
\end{itemize}

Since~$S^3$ is connected, it follows that $\Gamma$ is connected.

Since each of $\{\widetilde{C}, \dd V_1, \dd Y\}$ splits~$S^3$, it follows that each edge of~$\Gamma$ is a bridge.

Therefore, $\Gamma$ is a tree.

Since we assume that $C$ and $\gamma_Y$ lie in distinct components of~$\Delta\sm\gamma_1$, it follows that $\widetilde{C}$ and $\dd Y$ lie in distinct components of~$S^3\sm \dd V_1$. This means that the $\dd V_1$-edge separates 
the edges corresponding to~$\widetilde{C}$ and $\dd Y$.

Therefore, $\Gamma$ is a linear tree. This implies in particular that 
$Y(C)\supset V_1\supset Y$.
Then $Y(C)$ contains~$L$ because $V_1$ contains~$L$.
This implies the claim.

\begin{claim}
\label{cl:sum}
If a $\theta$-graph with the edges $\alpha_1$, $\alpha_2$, and $\alpha_3$ is embedded in~$\inr(\Delta)$ such that $\gamma_Y$ and the interior of the edge $\alpha_2$ lie in distinct components of~$\Delta\sm (\alpha_1\cup\alpha_3)$,
then the knot type represented by $Y(\alpha_1\cup\alpha_3)$ is a connected sum of the knot types represented by $Y(\alpha_1\cup\alpha_2)$ and $Y(\alpha_2\cup\alpha_3)$.
\end{claim}

We observe that any section of the lift of~$\theta$ gives an embedding $f\colon\theta\to S^3$ such that for any $i\neq j\in\{1,2,3\}$ the curve $f(\alpha_i\cup\alpha_j)$ is isotopic in $Y(\alpha_i\cup\alpha_j)$ to a core curve of $Y(\alpha_i\cup\alpha_j)$. This implies that $f(\alpha_i\cup\alpha_j)$ represents the same knot type as $Y(\alpha_i\cup\alpha_j)$ does. 
Now, let $a_2$ be a simple arc in~$\Delta$ with endpoints at~$\gamma_Y$ such that $a_2\cap\theta=\alpha_2$.
Since the lifts of the endpoints of~$a_2$ are meridians of~$Y$, 
it follows that there exist a pair of disjoint meridional disks in~$Y$ completing the lift of~$a_2$ to a sphere~$S^2$ in~$S^3$.
We see that $S^2$ contains $f(\alpha_2)$ and separates $f(\alpha_1)$ and $f(\alpha_3)$. This clearly implies the claim.

Now, we deduce from the above claims that $K$ is composite.
Let $\Delta_1$ be the connected component of~$\Delta\sm\gamma_1$ that does not contain~$\gamma_Y$. 
Claim~\ref{cl:essential-base} implies that $\Delta_1$ contains at least two components of~$\dd\Delta$.
Let $C_1$, \dots, $C_k$, $k\ge 2$, denote these components and let $K_1$, \dots, $K_k$ denote the knot types represented by the solid tori $Y(C_1)$, \dots, $Y(C_k)$, respectively. 
Claim~\ref{cl:non-sep-in-gammai} implies that each of $C_1$, \dots, $C_k$ is non-$L$-separating.
Then Claim~\ref{cl:non-sep-knotted} implies that each of $Y(C_1)$, \dots, $Y(C_k)$ is nontrivially knotted.
We observe that Claims~\ref{cl:Ygamma12} and~\ref{cl:sum} imply that the knot type~$K$ of~$Y(\gamma_1)=V_1$ is a connected sum of~$K_1$, \dots, $K_k$.
Since $k\ge 2$ and each of~$K_1$, \dots, $K_k$ is nontrivial, it follows that $K$ is composite.


\section{Proof of Assertion (iv) of Proposition~\ref{lem:main}}

\subsection{The case of multicomponent links}

We show that if $L$ is a prime non-split link with at least two components, then there exists a reliable $K$-en\-tangle\-ment $L'$ for~$L$ with $\crn(L')\le \crn(L)+6\crn(K)$. 
(Then assertions (i) and (ii) will imply that this $K$-en\-tangle\-ment is non-split and prime.)

Let $D_L$ be a minimal diagram of~$L$ in the unit sphere 
\begin{equation}
\label{eq:sphere}
S^2=\{(x,y,z)\in\R^3\colon \sqrt{x^2+y^2+z^2}=1\}.
\end{equation}
Since $L$ has at least two components, it follows that there exists a disk $\delta\subset S^2$ such that the boundary $\dd\delta$ intersects $D_L$ transversely in four points and the intersection $\delta\cap D_L$ consists of two simple disjoint arcs corresponding to distinct components of~$L$.
Applying a homeotopy of~$S^2$, we can assume without loss of generality that 
\begin{equation}
\label{eq:Splus}
\delta=S^2_+:=\{(x,y,z)\in S^2\colon x\ge 0\}.
\end{equation}
Furthermore, we can carry out an additional homeotopy of~$S^2$ (preserving~$\delta$ setwise) that moves $D_L$ to a position where $D_L$ is contained in the annulus
\begin{equation}
\label{eq:A}
A:=\{(x,y,z)\in S^2\colon |z|\le 1/100\} 
\end{equation}
and the intersection of $D_L$ with the semiannulus
\begin{equation}
\label{eq:Aplus}
A_+:=\{(x,y,z)\in A\colon x\ge 0\} 
\end{equation}
is the pair of semicircles
\begin{equation}
\label{eq:ddAplus}
(\dd A)_+:=\{(x,y,z)\in \dd A\colon x\ge 0\}.
\end{equation}

Now, let $D_K$ be a smooth minimal diagram of~$K$ in the plane~$\R^2\subset \R^3$.
If $D_K$ has nonzero writhe, we transform $D_K$ into a diagram $D'_K$ with zero writhe, by a series of type~I Reidemeister moves.
Since the (absolute value of the) writhe of a knot diagram does not exceed the number of crossings of the diagram, we can assume that 
$$
\crn(D'_K)\le 2\crn(D_K)=\crn(K).
$$

Next, let~$A_K$ be a knotted smooth embedded annulus in~$\R^3$ such that 
\begin{itemize}
\item[--] each component of~$\dd A_K$ represents~$K$;

\item[--] the projection~$\R^3\to\R^2$ restricts to a smooth map~$\pi\colon A_K\to\R^2$; 

\item[--] a component of~$\dd A_K$ smoothly projects to~$D_K$;

\item[--] the set of multiple points of~$\pi\colon A_K\to\R^2$ consists of $\crn(D'_K)$ domains; each of these domains is a geometric lozenge in~$\R^2$ containing a crossing of~$D'_K$; all multiple points are double points.
\end{itemize}

Let $f\colon A\to A_K$ be a homeomorphism.
We observe that since $D'_K$ has zero wreathe, it follows by construction that the components of~$\dd A_K$ has zero linking number. 
Therefore, the link~$L'$ obtained from the diagram~$f(D_L)$ via natural pushing out into a neighborhood of~$A_K$ is a (reliable) $K$-en\-tangle\-ment for~$L$.

Applying a homeotopy of~$A_K$, we can move $f(A_-)$, where 
$$
A_-:=\{(x,y,z)\in A\colon x\le 0\},
$$
to a position where $\pi\colon A_K\to\R^2$ is injective on~$f(A_-)$.
We observe that $\pi(f(D_L))$ yields a diagram, say $D_{L'}$, of~$L'$.
Since $\pi$ is injective on~$f(A_-)$
and the intersection of $D_L$ with $A_+$ is $(\dd A)_+$,
it follows that the number of crossings in $D_{L'}$ is
$$
\crn(D_L)+4\crn(D'_K)\le \crn(L)+8\crn(K).
$$
Furthermore, every four crossings of~$D_{L'}$ corresponding to a crossing of~$D'_K$ that appeared as the result of type~I Reidemeister move needed to zero the writhe, obviously can be reduced to a pair of crossings. This implies that we have 
$$
\crn(L')\le \crn(L)+6\crn(K).
$$
Thus, $L'$ is a reliable (hence non-split and prime by assertions (i) and (ii)) $K$-en\-tangle\-ment for~$L$ with $\crn(L')\le \crn(L)+6\crn(K)$, as required.

\subsection{The case of knots}

We show that if $K$ is nontrivial knot and $P$ is a $\frac23$-regular prime knot, then there exists a \underline{prime} $K$-en\-tangle\-ment $P'$ for~$P$ with $$\crn(P')\le \crn(P)+6\crn(K).$$ 
The proof is similar to that in the case of multicomponent links.
Let $D_P$ be a minimal diagram of~$P$ in the unit sphere~\eqref{eq:sphere}. 
Since~$P$ is $\frac23$-regular, it follows from
\cite[Proposition~8.1 and Definition~7.1]{Mal18}
that
there exists a $2$-disk $\delta\subset S^2$ such that:
\begin{itemize}
\item the boundary $\partial \delta$ intersects $D_P$ transversely in four points;
\item the intersection $\delta\cap D_P$ consists of two simple disjoint arcs; 
\item the $2$-string tangle (say, $(B,t)$) represented by the diagram $d \cap D_P$, where $d$ stands for the complementary disk $S^2\setminus \inr(\delta)$, is locally trivial.
\end{itemize}
Applying a homeotopy of~$S^2$, we can assume without loss of generality that 
$\delta=S^2_+$ (see~\eqref{eq:Splus}). 
Furthermore, following the above proof in the case of multicomponent links,
we can carry out an additional homeotopy of~$S^2$ (preserving~$\delta$ setwise) that moves $D_P$ to a position where $D_P$ is contained in the annulus~$A$ (see~\eqref{eq:A}) and the intersection of $D_P$ with the semiannulus
$A_+$ (see~\eqref{eq:Aplus})
is the pair of semicircles
$(\dd A)_+$ (see~\eqref{eq:ddAplus}).
Then we repeat the proof of the case of multicomponent links verbatim (except replacing $L$ with $P$ and $L'$ with $P'$) and construct a knot~$P'$ with $\crn(P')\le \crn(P)+6\crn(K)$.
The difference is that $P$ has a single component so that~$P'$ is not a reliable $K$-en\-tangle\-ment. 
In particular, assertion~(ii) is not applicable.
Thus, to complete the proof it remains to show that 
\begin{enumerate}
\item $P'$ is a $K$-en\-tangle\-ment for~$P$;
\item $P'$ is prime.
\end{enumerate}

In order to show that $P'$ is a $K$-en\-tangle\-ment for~$P$, it is enough to check that if $V_A$ is a solid torus obtained as a tubular neighborhood of~$A$ in~$\R^3$, then a representative of $P$ obtained via natural pushing out of~$D_P$ into~$V_A$ has wrapping number~$2$ in~$V_A$. This readily follows by Lemma~\ref{lem:wrapping} because the condition that the $2$-string tangle $(B,t)$ represented by the diagram $d \cap D_P$ is locally trivial implies that $d \cap D_P$ is connected (the option where $d \cap D_P$ is locally trivial being a pair of disjoint simple arcs is excluded by the condition that $P$ is a nontrivial knot).

The primeness of $P'$ follows from results obtained in~\cite{Mal18}. 
We observe that, by construction, $P'$ is a sum of a $K$-knotted cable tangle  and the tangle $(B, t)$, which is either prime or trivial (being locally trivial). Each nontrivial cable tangle is prime (see \cite[Lemma~6.2]{Mal18}). If $(B,t)$ is prime, then $P'$ is prime by \cite[Theorem~1]{Lick81}. If $(B,t)$ is trivial, then $P'$ is prime by \cite[Lemma~6.3]{Mal18} (Lemma~6.3 of \cite{Mal18} implies that $P'$ is either prime or trivial if $(B, t)$ is trivial; however, $P'$ is nontrivial being a satellite knot).

\subsection{An auxiliary lemma on wrapping numbers}

\begin{definition}[Wrapping numbers of curves]
We say that a piecewise smooth curve $\gamma\colon S^1\to F$ in a surface~$F$ is \emph{regular} if $\gamma$ has only finitely many multiple points, all these multiple points are double points, and the intersections in all these multiple points are transverse. By the \emph{wrapping number} of a regular curve~$\gamma=\gamma(S^1)$ in an annulus~$A$ we mean the minimum number of intersections of~$\gamma$ with any arc in~$A$ with endpoints at distinct components of~$\dd A$ that contains no multiple points of~$\gamma$.
\end{definition}

Let $A$ be an annulus and let $N\colon S^1\to V$ be a knot in the solid torus~$V=A\times[0,1]$. 
We denote by $p$ the natural projection 
$$
p\colon A\times[0,1]\to A, \quad (x,t)\mapsto x.
$$  
We say that a knot $N\colon S^1\to V$ is in a \emph{regular} position (with respect to~$p$) if $p\circ N$ is a regular curve.
Two knots $N_1$ and $N_2$ in~$V$ are \emph{equivalent} if there exists an ambient isotopy of~$V$ moving~$N_1$ to~$N_2$.
We say that a knot $N\colon S^1\to V$ is in a \emph{minimal} position (with respect to~$p$) if $N$ is in a regular position and the number of double points in~$p(N)$ is minimal over all knots in regular positions equivalent to~$N$.
Let $N\colon S^1\to V$ be a knot in a regular position.
Obviously, the wrapping number of~$p(N)$ in~$A$ is greater or equal to the wrapping number of~$N$ in~$V$.

\begin{lem}
\label{lem:wrapping}
If the wrapping number of~$p(N)$ in~$A$ equals~$2$ and $N$ is in a minimal position, then the wrapping number of~$N$ in~$V=A\times[0,1]$ equals~$2$ as well.
\end{lem}

\begin{proof}
Suppose to the contrary that the wrapping number of~$N$ in~$V$ is less than~$2$. 
The standard parity argument implies that the wrapping number of~$N$ in~$V$ is even. Thus, this number is zero. 
Then $V$ contains a ball~$B$ containing~$N$. 
Let $\widetilde{V}$ and $\widetilde{A}$ be the universal covering spaces of~$V$ and~$A$, respectively, and let $\tilde{p}\colon \widetilde{V}\to \widetilde{A}$ be a lift of~$p$.
Let $\widetilde{B}'$ be a component of the lift of~$B$ in~$\widetilde{V}$, and let $\widetilde{N}'$ be the component of the lift of~$N$ contained in~$\widetilde{B}'$. 

Since the wrapping number of~$p(N)$ in~$A$ equals~$2$, it follows that $A$ contains a simple arc~$a$ with endpoints at distinct components of~$\dd A$ such that $a$ intersect~$p(N)$ at precisely two points, these points are regular points of~$p(N)$ and the intersections are transverse. 
Let $\alpha$ denote the disk~$a\times[0,1]$, 
let $N_1$ and $N_2$ be the components of~$N\sm\alpha$, 
let $\tilde{\alpha}'$ be the component of the lift of~$N$ intersecting~$\widetilde{N}'_1$, and let $\widetilde{N}'_1$ and $\widetilde{N}'_2$ be the lifts of $N_1$ and $N_2$, respectively, contained in $\widetilde{N}'$.

We observe that $\tilde{\alpha}'$ splits $\widetilde{V}$ in two components and that $\widetilde{N}'_1$ and $\widetilde{N}'_2$ lie in distinct components of~$\widetilde{V}\sm \tilde{\alpha}'$. Consequently, the projections $\tilde{p}(\widetilde{N}'_1)$ and $\tilde{p}(\widetilde{N}'_2)$ are disjoint.

At the other hand, we easily see that the curve $\tilde{p}(\widetilde{N}')$ provides a diagram for the knot~$\widetilde{N}'$ in~$\inr(\widetilde{B}')$.
Since the pairs $(\widetilde{B}',\widetilde{N}')$ and $(B,N)$ are homeomorphic, it follows that, for each diagram~$D$ for $\widetilde{N}'$ in~$\inr(\widetilde{B}')$ there exists a homeotopy of $B$ that moves $N$ to a position $N_*$ such that the curve $p(N_*)$ is combinatorially equivalent to~$D$. This means in particular that the number of crossings of $\tilde{p}(\widetilde{N}')$ is not less than that of~$p(N)$ because the former supposed to be minimal possible for the ambient isotopy type of~$N$ in~$V$.

The number of crossings in~$p(N)$ is the sum of the numbers of crossings in $p(N_1)$, in $p(N_2)$, and between $p(N_1)$ and $p(N_2)$.
The number of crossings in~$\tilde{p}(\widetilde{N}')$ is the sum of the numbers of crossings in $\tilde{p}(\widetilde{N}'_1)$, in $\tilde{p}(\widetilde{N}'_2)$, and between $\tilde{p}(\widetilde{N}'_1)$ and $\tilde{p}(\widetilde{N}'_2)$. Since the number of crossings in~$\tilde{p}(\widetilde{N}'_i)$, $i=1,2$, is not less than that in $p(N_i)$, it follows that $p(N_1)$ and $p(N_2)$ are disjoint.
However, in this case there exists a simple arc with endpoints at distinct components of~$\dd A$ that does not intersect~$p(N)$.
This implies that the wrapping number of~$p(N)$ in~$A$ is~$0$.
The obtained contradiction proves the lemma. 
\end{proof}

\section{Proof of Proposition~\ref{pr:NxRegular}}

We prove that the proportion of non-$\frac34$-regular prime knots among all of the prime non-split links of $n$ or fewer crossings converges to $0$ as $n$ approaches infinity. Our proof uses the concept of diagrammatically prime knots and links.

\begin{definition}[Diagrammatically prime knots and links]
A link diagram~$D$ in~$\R^2$ is called \emph{prime} if $D$ is connected and if any circle in~$\R^2$ intersecting~$D$ transversely in two points cuts a simple subarc from~$D$.
We say that a link~$L$ is \emph{weakly diagrammatically prime} if a minimal plane diagram representing~$L$ is prime.
We say that $L$ is \emph{strongly diagrammatically prime} if all of the minimal plane diagrams of~$L$ are prime.
\end{definition}

\begin{remark}
Obviously, each prime non-split link is strongly diagrammatically prime.
The conjecture that the crossing number of knots is additive with respect to connected sum is equivalent to the conjecture that no composite knot is strongly diagrammatically prime.
\end{remark}

\begin{lem}
\label{lem:NRPKSDPK-1}
If there exists a non-$x$-regular prime knot~$P$ with ${0<x\le 1}$
then there exists a strongly diagrammatically prime composite knot~$K$ with $\crn(K)< x\cdot\crn(P)$ such that $P$ is a factor of~$K$.
\end{lem}

\begin{proof}
We say that a knot~$K$ is \emph{$P_x$-special} if $P$ is a factor of $K$ and $\crn(K)< x\cdot\crn(P)$.
If $P$ a non-$x$-regular knot, then the set of $P_x$-special knots is nonempty by definition. 
Let $K$ be a $P_x$-special knot with minimal crossing number in the set of $P_x$-special knots.
Since we have $x\le 1$, it follows that $K\neq P$ and hence $K$ is composite.
If there exists a composite minimal diagram of~$K$,
then $K$ is a sum of two nontrivial knots~$K_1$ and $K_2$ such that $\crn(K_1)< \crn(K)$ and $\crn(K_2)< \crn(K)$. 
Since $P$ is a factor of $K$, it follows by the Unique Factorization Theorem by~\cite{Schu49} that $P$ is a factor of either~$K_1$ or~$K_2$.
Without loss of generality we can assume that $P$ is a factor of~$K_1$.
Since $\crn(K_1)< \crn(K)<x\cdot\crn(P)$, it follows that $K_1$ is $P_x$-special, which contradicts the assumption that $K$ is a $P_x$-special knot with minimal crossing number.
This contradiction shows that all diagrams of~$K$ are prime.
This completes the proof.
\end{proof}


\begin{lem}
\label{lem:NRPKSDPK-2}
Let $\{N_xRPK_r\}$ denote the set of all non-$x$-regular prime knots of $r$ or fewer crossings, and let $\{SDPK_r\}$ denote the set of all strongly diagrammatically prime knots of $r$ or fewer crossings. 
Then for any $r>0$ and $x\in(0,1]$ we have
\begin{equation}
\label{eq:NRPKSDPK}
\card\{N_xRPK_r\}~\le~\frac{\card\{SDPK_{rx}\}}{152\cdot x}.
\end{equation}
\end{lem}

\begin{proof}
If $\{N_xRPK_r\}$ is empty, then~\eqref{eq:NRPKSDPK} is trivially true.
If $\{N_xRPK_r\}$ is not empty, then Lemma~\ref{lem:NRPKSDPK-1} implies that there exists a function
$$
f\colon \{N_xRPK_r\}\to \{SDPK_{rx}\}
$$
sending $P\in\{N_xRPK_r\}$ to a strongly diagrammatically prime composite knot $f(P)$ with factor~$P$ such that  
\begin{equation}
\label{eq:f-def}
\crn(f(P))~<~ x\cdot\crn(P).
\end{equation}
A~result of~\cite{La09} states that for any set of (oriented) knots $K_1$, $\dots$, $K_n$ in the $3$-sphere we have  
\begin{equation}
\label{eq:Lackenby152}
\frac{\crn(K_1)+\dots+\crn(K_n)}{152}~\le~\crn(K_1\sharp\dots\sharp K_n).
\end{equation}
If $K$ lies in the codomain~$f(\{N_xRPK_r\})$, let $P$ be a knot in $f^{-1}(K)$ 
having the smallest crossing number 
among the elements of $f^{-1}(K)$.  
Then~\eqref{eq:Lackenby152} implies that
\begin{equation}
\label{eq:Lack-cor}
\frac{\card(f^{-1}(K))\cdot\crn(P)}{152}~\le~\crn(K).
\end{equation}
We observe that \eqref{eq:f-def} and~\eqref{eq:Lack-cor} imply that for each knot $K$ in $\{SDPK_{rx}\}$ we have
\begin{equation*}
\card(f^{-1}(K))~<~152\cdot x.
\end{equation*}
This implies the required inequality.
\end{proof}

Lemma~\ref{lem:NRPKSDPK-2} implies that for any $r>0$ and $x\in(0,1]$ we have
\begin{multline}
\label{eq:NRPKSDPK-4}
\limsup_{r\to\infty}\sqrt[r]{\card\{N_xRPK_r\}}
~\le~
\limsup_{r\to\infty}\sqrt[r]{\frac{\card\{SDPK_{rx}\}}{152\cdot x}}\\
~=~
\limsup_{r\to\infty}\sqrt[r]{\card\{SDPK_{rx}\}}~=~
\left(\limsup_{r\to\infty}\sqrt[rx]{\card\{SDPK_{rx}\}}\right)^x\\
~=~
\left(\limsup_{r\to\infty}\sqrt[r]{\card\{SDPK_{r}\}}\right)^x.
\end{multline}

We observe that
$$
\limsup_{r\to\infty}\sqrt[r]{\card\{SDPK_{r}\}}
~\le~
\limsup_{r\to\infty}\sqrt[r]{r\cdot\card\{sdpk_{r}\}}
~=~
\limsup_{r\to\infty}\sqrt[r]{\card\{sdpk_{r}\}},
$$
where $\{sdpk_{r}\}$ 
is the set of all strongly diagrammatically prime knots of precisely~$r$ crossings. 
Estimates for $\limsup_{r\to\infty}\sqrt[r]{\card\{sdpk_{r}\}}$ can be derived from the papers~\cite{Wel92, STh98, St04}. These papers contain estimates for the number of prime link diagrams and for the number of prime link diagrams modulo flypes. These estimates can be used for prime links and for diagrammatically prime links as well.
Results of~\cite{St04} (already used above) imply that
$$
\limsup_{r\to\infty}\sqrt[r]{\card\{sdpk_{r}\}}~\le~\frac{\sqrt{13681}+91}{20}~<~10.4.
$$
Therefore, we have
$$
\limsup_{r\to\infty}\sqrt[r]{\card\{N_xRPK_r\}}
~<~(10.4)^x.
$$
In particular, we have
$$
\limsup_{r\to\infty}\sqrt[r]{\card\{N_{\frac34}RPK_r\}}
~<~(10.4)^{\frac34}~<~5.8.
$$
This implies the required statement because results of~\cite{STh98} say that
$$
\liminf_{r\to\infty}\sqrt[r]{P_r}~\ge~\frac{\sqrt{21001}+101}{40}~>~6.14,
$$
where $P_r$ stands for the number of prime non-split links of $n$ or fewer crossings.

\end{document}